\documentclass{amsart}
\usepackage{lmodern}
\usepackage[T1]{fontenc}
\usepackage[utf8]{inputenc}
\usepackage{amsmath}
\usepackage{amssymb}
\usepackage{amsthm}
\usepackage{enumerate}
\usepackage{bbm}
\usepackage{multicol}
\usepackage{amsfonts}
\usepackage[colorlinks=true,citecolor=red,linkcolor=blue]{hyperref}

\newtheorem{theorem}{Theorem}[section]
\newtheorem{proposition}[theorem]{Proposition}
\newtheorem{corollary}[theorem]{Corollary}
\newtheorem{lemma}[theorem]{Lemma}

\newtheorem{rmk}[theorem]{Remark}

\newtheorem*{theorem*}{Theorem}
\newtheorem*{corollary*}{Corollary}

\theoremstyle{definition}
\newtheorem{definition}[theorem]{Definition}

\newcommand{\al}{\omega_1}
\renewcommand{\P}{\mathbb{P}}
\newcommand{\Q}{\mathbb{Q}}

\newcommand{\seb}{\subseteq}
\newcommand{\setm}{\setminus}
\newcommand{\es}{\emptyset}
\newcommand{\tleq}{\sqsubseteq}

\newcommand{\res}{\restriction}
\newcommand{\tup}[1]{\langle#1\rangle}
\newcommand{\PFA}{\mathrm{PFA}}
\newcommand{\CH}{\mathrm{CH}}
\newcommand{\MRP}{\mathrm{MRP}}
\newcommand{\OGA}{\mathrm{OGA}}

\newcommand{\dt}[1]{\dot{#1}}
\newcommand{\forces}{\Vdash}
\newcommand{\ck}[1]{\check{#1}}
\DeclareMathOperator{\dom}{dom}
\DeclareMathOperator{\ran}{ran}
\DeclareMathOperator{\hht}{ht}

\title{New consequences of PFA($T^*$)}
\date{October 26, 2025}

\author{Carlos Martínez-Ranero}
\address[C. Martínez-Ranero]{Department of Mathematics, Universidad de Concepción, Concepción, Chile}
\email{cmartinezr@udec.cl}
\urladdr{www2.udec.cl/\textasciitilde cmartinezr}

\author{Lucas Polymeris}
\address[L. Polymeris]{Department of Mathematics, Universidad de Concepción, Concepción, Chile}
\email{l.polymeris@proton.me}

\thanks{The first named author was partially supported by Proyecto VRID-Investigación  No. 220.015.024-INV}

\begin{document}

\begin{abstract}
  Let $T^*$ be an almost Suslin tree, that is, an Aronszajn tree with no stationary antichains. Krueger \cite{Krueger2020} introduced a forcing axiom, $\PFA(T^*)$, for the class of proper forcings that preserve that $T^*$ is almost Suslin. He showed that $\PFA(T^*)$ implies several well-known consequences of the Proper Forcing Axiom ($\PFA$), including Suslin's Hypothesis and the P-ideal dichotomy.
 We extend this list by proving that $\PFA(T^*)$ also implies the Mapping Reflection Principle ($\MRP$) and the Open Graph Axiom ($\OGA$). Additionally, we show that $\PFA(T^*)$ implies that all special Aronszajn trees are club-isomorphic, but it does not imply that all almost Suslin trees are club-isomorphic.
\end{abstract}

\maketitle

An Aronszajn tree is called \emph{almost Suslin} if it has no stationary
antichains. For such a tree $T^*$, Krueger
introduced in \cite{Krueger2020}  the forcing axiom $\PFA(T^{*})$, a relativized version of $\PFA$ which applies to all proper forcings preserving that $T^{*}$ is almost Suslin.

Krueger showed that $\PFA(T^*)$ already implies many of the well-known consequences
of the proper forcing axiom, including the Suslin's Hypothesis,
the $P$-ideal dichotomy,  and the failure of a very weak form of club
guessing. He raised the question of whether further familiar consequences of $\PFA$ also follow from $\PFA(T^*)$, specifically: the mapping reflection principle ($\MRP$), the open graph axiom ($\OGA$) and the non-existence of weak Kurepa trees.

A second question posed in \cite{Krueger2020}, attributed to Moore, asks whether $\PFA(T^*)$
 implies that all normal almost Suslin trees are club isomorphic. This is a natural analogue of the  fact that $\PFA$ implies all normal Aronszajn trees are club-isomorphic; a positive answer would suggest that models of
$\PFA(T^*)$ are somewhat canonical.

In this work, we address these questions. We show that $\PFA(T^{*})$  imply both $\MRP$ and $\OGA$, by constructing suitable modifications of the standard proper forcings for these principles. In contrast, we answer Moore's question negatively: starting from a model of $\diamondsuit$, we construct two almost Suslin trees $T$ and $S$ that are not club-isomorphic and show they remain so after forcing to obtain a model of $\PFA(T+S)$. As a positive counterpart, we prove that $\PFA(T^{*})$ does imply that any two normal special Aronszajn trees are club-isomorphic, thus establishing a partial analogue of the situation under $\PFA$.

We note that subsequent work by Lambie-Hanson and Stejskalová \cite{Lambie2025} has established that $\PFA(T^{*})$ implies the non-existence of weak Kurepa trees. Furthermore, \cite[Fact 6.19(4)]{Lambie2025} --a fact originally remarked in \cite{Krueger2020}-- can be used to show that $\PFA(T^{*})$ implies the non-existence of $\omega_2$-Aronszajn trees, via the standard argument from $\PFA$.

The structure of our paper is as follows.
After establishing the necessary preliminaries in Section \ref{sec:prelim}, we proceed to demonstrate the main combinatorial consequences of $\PFA(T^*)$. Specifically, Section \ref{sec:mrp} is dedicated to proving that the axiom implies the Mapping Reflection Principle ($\MRP$), and Section \ref{sec:oga} proves that it implies the Open Graph Axiom ($\OGA$). We then turn to structural consequences for trees in Section \ref{sec:club-iso}, where we resolve the question of whether $\PFA(T^*)$ implies that all almost Suslin trees are club-isomorphic.

\section{Preliminaries}\label{sec:prelim}
A  set-theoretic tree is a partial order set $(T,<_{T})$  such that
for any node $t \in T$, the set of predecessors  $\{s \in T : s <_{T} t\}$   is   well-ordered by $<_T$. The \emph{height} of $t$, written $\hht_{T}(t)$, is the unique ordinal representing the  order-type of the set of its predecessors.  For an ordinal $\alpha$, the
\emph{$\alpha$-th} level of $T$, written $T_{\alpha}$, consists
of the elements of $T$ of height $\alpha$. The \emph{height of $T$}
is the least ordinal $\alpha$ for which $T_{\alpha}$ is empty.

A \emph{branch} of $T$ is any downwards closed
chain of $T$, and a branch is called \emph{cofinal} if its height is equal to the height
of $T$. For each $t \in T$ and $\xi \le \hht(t)$, we let
$t\res \xi$ to be the unique $s \in T$ of height $\xi$ such that
$s \le_{T} t$. If the tree is clear from the context, then $t \perp s$ denotes the fact  that $t$ and $s$ are incomparable in the poset $(T,\le_{T})$. The tree $T$ is said to be \emph{well-pruned} if
for all $t \in T$, the subtree $\{s \in T : s \not\perp t\}$ has the same height
as $T$.

For incomparable $t,s \in T$, we let
$\Delta(t,s)$ to be the least ordinal  $\alpha$ such that
$t \res \alpha \neq s \res \alpha$. We leave $\Delta$ undefined
if $t$ and $s$ are comparable. A tree is called
\emph{Hausdorff} if $\Delta(t,s)$ is always a successor ordinal.

An $\omega_{1}$-tree is a tree of height $\omega_{1}$ with
all its levels countable. An  $\al$-tree is called an \emph{Aronszajn
tree} if it has no uncountable chains, equivalently, no cofinal branches.
An $\al$-tree is called a \emph{Suslin tree} if it has no uncountable chains and no uncountable
antichains. A typical argument shows that if $T$ is a well-pruned
$\omega_{1}$-tree, then it is Suslin iff it has no uncountable
antichains. A tree is called \emph{normal} if it is Hausdorff,
well-pruned, and
every node has $\aleph_{0}$ immediate successors.

\begin{definition} Let $T$ be a tree with height $\omega_1$. A subset
  $A\subseteq T$ is called a \emph{stationary antichain} if $A$ is an
  antichain and the set $\{\hht_T(t): t\in A\}$ is a stationary subset
  of $\omega_1$.
\end{definition}

\begin{definition} An \emph{almost Suslin tree} is any Aronszajn tree
  without stationary antichains.
\end{definition}

Fix an Aronszajn tree $T$.
The remaining definitions and lemmas of this section are all from
\cite{Krueger2020}. We define a relation $<_{T}^-$ by letting $x
<_{T}^- y$ if $\hht_{T} (x) < \hht_{T}(y)$ and $x \not<_{T} y$.

\begin{definition} Let $N$ be a countable set such that $\delta := N
  \cap \omega_1$ is an ordinal. A node $x \in T_\delta$ is said to be
  \emph{$(N,T)$-generic} if for any set $A \subseteq T$ which is a
  member of $N$ and any relation $R\in\{<_{T},<_{T}^- \}$, if $x\in A$,
  then there exists $yRx$ such that $y \in A$.
\end{definition}

\begin{lemma}\label{lemmaMRP}{\rm (\cite[Lemma 1.4]{Krueger2020})} Let $\theta$ be a regular cardinal with
  $T \in H(\theta)$, and $N$ a countable elementary submodel of
  $H(\theta)$ containing $T$ as an element. Let $\delta := N \cap
  \omega_1$. Assume that $x \in T_\delta$ is $( N , T )$-generic. Let $B
  \subseteq T$ be in $N$ and $R \in \{ <_{T}, <_{T}^- \}$. If for all
  $y R x$, $y \in B$, then $x \in B$.
\end{lemma}

As mentioned in the introduction, we will be mainly interested in the
case that $T$ is almost Suslin, one of the reasons being the
following.

\begin{lemma}{\rm (\cite[Lemma 1.5]{Krueger2020})} Let $\theta$ be a regular cardinal such that $T \in
  H(\theta)$.  Let $N$ be a countable elementary submodel of $H(\theta)$
  containing $T$ as an element.  If $T$ is almost Suslin, then for all
  $x \in T_{N\cap\al}$, $x$ is $(N, T)$-generic.
\end{lemma}

\begin{definition} Let $T$ be an $\omega_1$-tree and $\mathbb{P}$ a
  forcing poset. Let $\theta$ be a regular cardinal such that $T$ and
  $\mathbb{P}$ are members of $H(\theta)$ and $N$ a countable elementary
  submodel of $H(\theta)$ containing $T$ and $ \mathbb{P}$ as
  elements. A condition $q \in \mathbb{P}$ is said to be \emph{$(N,
    \mathbb{P}, T )$-generic} if $q$ is $(N,\mathbb{P})$-generic, and
  whenever $x \in T_{N \cap \omega_1}$ is $(N, T )$-generic, then $q$
  forces that $x$ is $(N[\dot G_\mathbb{P}],T)$-generic.
\end{definition}

Assume that $q$ is $(N, \mathbb{P})$-generic. Then the following
property is easily seen to be equivalent to $q$ being
$(N,\P,T)$-generic: for any node $x \in T_{N\cap{\al}}$ which is
$(N,T)$-generic, a $\P$-name $\dot A \in N$ for a subset of $T$, and
relation $R \in \{<_{T},<_{T}^- \}$, if $r \leq q$ forces that
$x\in\dot A$, then there exists $y R x$ and $s \leq r$ such that $s$
forces that $y\in \dot A$.

\begin{definition} Let $T$ be an $\al$-tree and $\P$ a forcing
  poset. We say that $\P$ is \emph{$T$-proper} if for all large enough
  regular cardinals $\theta$ with $\P$ and $T$ in $H(\theta)$, there are
  club many $N$ in $[H(\theta)]^\omega$ such that $N$ is an elementary
  submodel of $H(\theta)$ containing $T$ and $\P$ as elements satisfying that
  for all $p \in N \cap \P$, there is $q \leq p$ which is $(N, \P, T)$-generic.
\end{definition}

\begin{definition} Define $\PFA(T)$ to be the statement that for
  any $T$-proper forcing poset $\P$, whenever $\langle D_\alpha :
  \alpha <\al \rangle$ is a sequence of dense subsets of $\P$, then
  there exists a filter $G$ on $\P$ such that for all $\alpha<\al, G\cap
  D_\alpha \ne \emptyset$.
\end{definition}

\begin{theorem}
  If $T^{*}$ is an almost Suslin tree, and there is a supercompact
  cardinal, then there is a forcing poset which forces
  $\PFA(T^{*})$ and $\kappa = \omega_{2}$.
\end{theorem}

For the rest of this paper $T^{*}$ will always denote an almost Suslin tree.

\section{Mapping reflection principle}\label{sec:mrp}
In this section we shall prove that  $\PFA(T^*)$ implies Moore's Mapping Reflection Principle. We recall the relevant definitions.

Let $X$ be an uncountable set. We equip $[X]^\omega$ with the Ellentuck  topology obtained by declaring the sets $[a,A]:=\{Y\in[X]^\omega : a\subseteq Y\subseteq A\}$ to be open for all $a\in[X]^{<\omega}$ and $A\in[X]^\omega$. Let $\theta$ be a regular cardinal such that $[X]^\omega\in H(\theta)$ and let $M$ be a countable elementary submodel of $H(\theta)$ containing $[X]^\omega$ as an element. We say that a subset $\Sigma\subseteq [X]^\omega$ is $M$-\emph{stationary } if $\Sigma\cap E\cap M\ne\emptyset$ for any club $E\subseteq [X]^\omega$ in $M$.

\begin{definition}
    An \emph{open stationary} set mapping $\Sigma$ is a map whose domain is a club subset of countable elementary submodels of $H(\theta)$ containing $[X]^\omega$ as an element and such that $\Sigma(M)\subseteq [X]^\omega$ is open (in the Ellentuck topology) and $M$-stationary for all $M$ in its domain.
\end{definition}
\begin{definition}
An open stationary set mapping $\Sigma$ \emph{reflects} if there exists a continuous $\in$-increasing sequence  $\langle N_\xi:\xi<\omega_1\rangle $  such that, for all $0<\nu<\omega_1$, $N_\nu\in \dom(\Sigma)$  there exists $\nu_0<\nu $ such that $N_\xi\cap X\in \Sigma(N_\nu)$ for all $\nu_0<\xi<\nu$ (observe that the last condition is only meaningful for limit ordinals). We say that $\langle N_\xi:\xi<\omega_1\rangle $ is a reflecting sequence for $\Sigma$.
\end{definition}

\begin{definition}
The Set Mapping Reflection Principle  $\MRP$ is the assertion that all open stationary set mappings reflect.
\end{definition}
 We shall prove that there is $T^*$-proper forcing which adds a reflecting sequence. There are several proper posets in the literature for adding such a sequence (see \cite{Moore2005}, \cite{Moh2021}). However, we were unable to prove that any of them are $T^*$-proper. Thus, we will prove that a  variant of Moore's forcing is as required.
\begin{definition}
    Let $\Sigma$ be an open stationary set mapping (with $X, \theta$ as before). Define $\mathbb{P}_\Sigma$ to be the forcing poset whose conditions are all pairs $p=(f_p  ,F_p)$ where  $f_p:\alpha_p+1\to \dom(\Sigma)$ is a continuous   $\in$-increasing map, $\alpha_p\in\al$  and $F_p$ is a countable subset of $X$ such that for all  $0<\nu\leq\alpha_p$, there exists $\nu_0<\nu $ such that $f_p(\xi)\cap X\in \Sigma(f_p(\nu))$ for all $\nu_0<\xi<\nu$, ordered by $(f_q, F_q) \leq (f_p, F_p) $ if $f_q$ is an extension of $f_p$, $F_p\subseteq F_q$ and $F_p\in f_q(\xi)$ for all $\alpha_p<\xi\leq \alpha_q$.
\end{definition}
 The next two lemmas will help us prove that $\mathbb{P}_\Sigma$ is $T^*$-proper.

\begin{lemma}\label{keymrp}
    Let $\theta^*$ be a  regular cardinal such that $\mathbb{P}_\Sigma, H(|\mathbb{P}_\Sigma|^+), \Sigma$ and $T^*$ are members of $H(\theta^*)$, and let $M$ be a countable elementary submodel of $H(\theta^*)$ containing $\mathbb{P}_\Sigma,H\left (|\mathbb{P}_\Sigma|^+\right ), \Sigma $ and $T^*$ as elements. Let $p\in M\cap \mathbb{P}_\Sigma, \dot A\in M $ be a $\mathbb{P}_\Sigma$-name for a subset of $T^*$ and $y\in M\cap T^*.$ Then either:
    \begin{enumerate}

    \item[{\rm (i)}] there exists $q\leq p$ in $M$ such that $f_q(\xi)\cap X\in \Sigma(M\cap H(\theta))$ for $\alpha_p<\xi\leq \alpha_q$ and $q\Vdash y\in \dot A$, or

   \item[{\rm (ii)}]  there exists $F\in [X]^\omega$ such that $(f_p, F_p\cup F)\Vdash y\notin \dot A.$
        \end{enumerate}

\end{lemma}
\begin{proof}
    Suppose (i) fails. Let $E$ denote the collection of all sets of the form $N=N^*\cap X$ where $N^*$ is a countable elementary submodel of $H(|\mathbb{P}_\Sigma|^+)$ containing $H(\theta), p, \dot A, y, T^* $ and $\mathbb{P}_\Sigma$ as elements. It follows that $E\subseteq [X]^\omega$ is a club in $M\cap H(\theta)$. Since $\Sigma(M\cap H(\theta))$ is open and $M\cap H(\theta)$-stationary there is an $N\in E\cap \Sigma(M\cap H(\theta))\cap M$ and there is an $F\in [N]^{<\omega}$ such that $[F,N]\subseteq \Sigma(M\cap H(\theta))$. We claim that $(f_p, F_p\cup F)\Vdash y\notin \dot A$.  To prove the claim. Suppose for a contradiction that this is not the case. Then there is a $r\leq (f_p, F_p\cup F)$ such that $r\Vdash y\in \dot A$. Hence, by elementarity, there is a $q\in N^*\cap \mathbb{P}_\Sigma$ such that $q\Vdash y\in \dot A$. By the definition of the order, it follows that $f_q(\xi)\in [F,N]\subseteq \Sigma(M\cap H(\theta))$ for all $\alpha_p<\xi\leq \alpha_q$. Thus, $q$ is a witness that (i) holds which is a contradiction.

\end{proof}
     \begin{lemma}\label{keymrp2}
    Let $\theta^*$ be a  regular cardinal such that $\mathbb{P}_\Sigma, H(|\mathbb{P}_\Sigma|^+), \Sigma$ and $T^*$ are members of $H(\theta^*)$, and let $M$ be a countable elementary submodel of $H(\theta^*)$ containing $\mathbb{P}_\Sigma,H\left (|\mathbb{P}_\Sigma|^+\right ), \Sigma $ and $T^*$ as elements. Let $\delta=M\cap \omega_1$. Suppose that $x\in T_\delta^*$, $p\in M\cap \mathbb{P}_\Sigma, \dot A\in M $ is a $\mathbb{P}_\Sigma$-name for a subset of $T^*$ and $R\in\{<_{T^*},<_{T^*}^-\}.$  Then either:
    \begin{enumerate}

    \item[{\rm (i)}] there exists $yRx$ and $q\leq p$ in $M$ such that $f_q(\xi)\cap X\in \Sigma(M\cap H(\theta))$ for $\alpha_p<\xi\leq \alpha_q$ and $q\Vdash y\in \dot A$, or

   \item[{\rm (ii)}]  there exists $F\in [X]^\omega$ such that $(f_p, F_p\cup F)\Vdash x\notin \dot A.$
        \end{enumerate}

\end{lemma}
\begin{proof}
    If (i) fails, then it follows, from Lemma \ref{keymrp}, that for all $yRx$ there exists $F_y\in[X]^\omega$ such that $(f_p, F_p\cup F_y)\Vdash y\notin \dot A$. Let $B$ be the set of $z \in T^*$ for which there exists a countable set $F\in [X]^\omega$ such that $(f_p,F_p \cup F)\Vdash z\notin \dot A$. Notice that $B \in M$ since it is definable from parameters in $M$. By assumption for all $y R x, y \in B$. Applying \ref{lemmaMRP}, it follows that $x \in B$. So there exists a countable set $F\in [X]^\omega$  such that $(f_p, F_p \cup F) \Vdash x\notin \dot A$.
\end{proof}
We are ready to prove the main Theorem of the section.
\begin{theorem}
    $\PFA(T^*)$ implies $\MRP$.
\end{theorem}
\begin{proof}
 We shall prove that $\mathbb{P}_\Sigma$ is $T^*$-proper and adds a reflecting sequence. To do this, fix a large enough regular cardinal $\theta^*$ and a countable elementary substructure $M$ of $H(\theta^*)$ such that  $\mathbb{P}_\Sigma, \Sigma, T^*$ and $H(|\mathbb{P}_\Sigma|^+)$ are all in $M$. Consider $p\in \mathbb{P}_\Sigma\cap M.$ Let $\delta=M\cap \omega_1.$
Let $\langle (D_n, x_n,\dot A_n,i_n):n<\omega\rangle$ be an enumeration of all $4$-tuples $(D,x,\dot A,i)$ where $D$ is a dense subset of $\mathbb{P}_\Sigma$ in $M$, $x\in T^*_\delta$, $\dot A$ is a $\P_\Sigma$-name for a subset of $T^*$ and $i\in 3$. We will construct recursively a partial function $G:\omega\to ([X]^\omega\cup T^*\restriction \delta)$ and a sequence of conditions $\langle p_n:n<\omega\rangle $     such that:
\begin{enumerate}[(i)]
    \item $p_{n+1}\leq p_n$;
    \item $p_{n+1}\in D_n$ whenever $i_n=0$;
    \item $p_n \in \mathbb{P}_\Sigma\cap M;$
    \item $\forall \alpha_{p_n}< \xi\leq \alpha_{p_{n+1}}$, $  f_{p_{n+1}}(\xi)\cap X\in \Sigma(H(\theta)\cap M);  $
    \item $G$ satisfies the following:
    \subitem{\rm (a)} If $G(n)\in T^*\restriction \delta,$ then $p_{n+1}\forces G(n)\in \dt{A}_n$, otherwise
    \subitem{\rm (b)}$(f_{p_n},F_{p_{n}}\cup G(n))\forces x_n\notin \dt{A}_n$.
\end{enumerate}
  Let $p_0:=p$. Suppose that $p_n$ and $G$ (up to $n$) have been defined. We proceed by cases. First suppose $i_n=0$. In this case $n\notin \dom(G)$. Let $E_n $ be the set of all elements of the form $N=N^*\cap X$ where $N^*$ is a countable elementary submodel of $H(|\mathbb{P}_\Sigma|^+)$   which contains $H(\theta), \P_\Sigma,  D_n, \dot A_n$, $T^*$ and $p_n$ as elements. Notice that $E_n\subseteq [X]^\omega$ is a club in $M\cap H(\theta)$. Using that $\Sigma(M\cap H(\theta))$ is open and $(M\cap H(\theta))$-stationary, we can find $N_n\in \Sigma(M\cap H(\theta))\cap E_n\cap M, $ and a finite set $F_n\subseteq N_n$ such that $[F_n,N_n]\subseteq \Sigma(M\cap H(\theta)). $  Let $N_n^*$ be a countable elementary submodel of $H(|\mathbb{P}_\Sigma|^+)$ witnessing that $N_n\in E_n$. Let $p_{n+1}$ be any extension of $ (f_{p_n},F_{p_n}\cup F_n)$ in $D_n\cap M$ (such an extension exists by elementarity). Next suppose, $i_n\ne 0$. If $i_n=1$, then set $R=<_{T^*}$ and otherwise set $R=<_{T^*}^-$.  By Lemma \ref{keymrp2}, there are two possibilities, either there is an $r\leq p_n$ in $M$ and $yRx$ such that  $f_r(\xi)\cap X\in \Sigma(H(\theta)\cap M)$ for $\alpha_{p_n}<\xi\leq \alpha_r$ and $r\Vdash y\in \dot A$. In this case set $G(n):=y$ and $p_{n+1}:=r$. Or there is a $G_n\in[X]^\omega$ such that $(f_{p_n},F_{p_n}\cup G_n)\Vdash x_n\notin \dot A_n$. In this case, set $p_{n+1}:=p_n$ and $G(n):=G_n$.

  Let $f_q:=\bigcup_{n\in \omega}f_{p_n}\cup \{(\alpha,M\cap H(\theta))\}$ where $\alpha:=\sup_{n\in\omega}\alpha_{p_n}$ and $$F_q:=\bigcup_{n\in\omega} F_{p_n}\cup \bigcup\{G(n) : G(n)\in [X]^\omega\}.$$ Let $q:=(f_q,F_q)$. We claim that $q$ is $(N,\P_\Sigma,T^*)$-generic.
First we check that $q$ is a condition. Notice that for each, $x\in M\cap X$  the set $D_x:=\{r\in \P_\Sigma: x\in \bigcup\ran( f_r)\}$ is an open dense set in $M$. Then it follows from the construction that $\bigcup_{\xi<\alpha}f_q(\xi)=M\cap H(\theta). $ We infer from clause (iii)  that $f_q(\xi)\cap X\in \Sigma(M\cap H(\theta))$ for $\alpha_{p}<\xi<\alpha$. Thus, $q$ is a  condition and by definition $q\leq p_n$ for $n\in\omega.$ By construction  $q\in D$ for every dense open subset in $M$. Hence, $q$ is $(M,\P_\Sigma)$-generic.
It remains to show that $q$ is $(M,\P_\Sigma,T^*)$-generic. To do this, let $D$ be a dense open set in $M$ (arbitrary), $x\in T^*_\delta$, $\dot A\in M$ a $\P_\Sigma$-name for a subset of $T^*$ and $R\in\{<_{T^*},<_{T^*}^-\}$. Set $i=1$ if $R=<_{T^*}$ and $i=2$ otherwise. Fix $n$ such that $(D_n,x_n,\dot A_n,i_n)=(D,x,\dot A,i)$. At stage $n$ of the construction we consider two possibilities. The first was the existence of  a $yRx$ and $r\leq p_n$ such that $r\Vdash y\in \dot A$. In this case $G(n)Rx$ and $p_{n+1}\Vdash G(n)\in \dot A$. Thus, $G(n)Rx$ and $q\Vdash G(n)\in \dot A$ as required. The second possibility is that there exists an element in $[X]^\omega$, called $G(n),$ such that $(f_{p_n},F_{p_n}\cup G(n))\Vdash x\notin \dot A $. Since $q\leq (f_{p_n},F_{p_n}\cup G(n))$ then it follows that $q\Vdash x\notin \dot A$.

We now show that $\P_\Sigma$ adds a reflecting sequence. We claim that, for each $\alpha<\al$, the set $D_\alpha:=\{p\in\P_\Sigma: \alpha\in \dom(p)\}$ is dense open.  To see this, notice that the sets $D_x:=\{p\in \P_\Sigma: x\in\bigcup \ran(p)\}$ are dense open for every $x\in H(\theta)$. It follows that $\mathbbm{1}\Vdash \check H(\theta)=\bigcup_{p\in \dot G}\ran(p)$. Since $\P_\Sigma$ is proper, then $\omega_1$ is preserved and hence,  $\mathbbm{1}\Vdash \check \al\subseteq \dom(\bigcup \dot G) $. This concludes the proof of the Theorem.
\end{proof}

\section{Open Graph Axiom}\label{sec:oga}
 The purpose of this section is to prove that $\PFA(T^*)$ implies  $\OGA$. We recall the relevant definitions.
\begin{definition}
    Let $(X,d)$ be a separable metric space. An open graph on $X$ is a structure $\mathcal{G} = (X, E)$ where the edge set $E$ is viewed as an open symmetric subset of $X^2 \setminus \{(x, x) : x \in X\}$.
\end{definition}

\begin{definition}
    The open graph axiom $\OGA$ is the statement that whenever $\mathcal{G} = (X, E)$ is an open graph on a separable metric space $(X,d)$, then either there is an uncountable subset $Y\in[X]^{\al}$ such that $[Y]^2\subseteq E$ or there is  decomposition $X=\bigcup_{n\in\omega}X_n$ with $[X_n]^2\cap E=\emptyset$ for all $n\in\omega$.

\end{definition}

Fix an open graph  $\mathcal{G} = (X, E)$ on a separable metric space $(X,d)$. We shall prove that, assuming $\PFA(T^*)$, one of the two alternatives of $\OGA$ holds for $\mathcal{G}$. If the second alternative holds, then there is nothing to do. Hence, suppose the second alternative fails. We will prove that there is a $T^*$-proper forcing which adds an uncountable subset $Y\subseteq X$ such that $[Y]^2\subseteq E$. There are several proper forcing in the literature that add such a set. We shall prove that a small variant of Todorcevic's forcing in \cite{Todorcevic2011} is as required.

Let $\mathcal{I}$ be the (proper) $\sigma$-ideal consisting of all subsets $Y$ of $X$ that are countably chromatic i.e. there is a decomposition $Y=\bigcup_{n\in\omega}Y_n$ with $[Y_n]^2\cap E=\emptyset$ for all $n\in\omega$. Furthermore, by removing a relatively open subset of $X$, we may assume that no nonempty  open subset of $X$ belongs to $\mathcal I$.

Let $\mathcal{H}:=\mathcal{I}^+$ denote the corresponding co-ideal i.e. the set of elements not in $\mathcal{I}$.
\begin{definition}
    For each $n>0$, the Fubini power $\mathcal{H}^n$ is the co-ideal on $X^n$ defined recursively by $\mathcal{H}^1:= \mathcal{H}$ and for $n > 1$,
$$\mathcal{H}^n=\{W \subseteq X^n :\{\overline{x} \in X^{n-1} :W_{\overline{x}}\in \mathcal{H}\}\in \mathcal{H}^{n-1}\},$$ where for $W \subseteq X^n$ and $\overline{x} \in X^{n-1}$,
$$W_{\overline{x}}= \{ y \in X : \overline{x}^\frown y \in W \}.$$
\end{definition}

 The following Lemma appears as \cite[Lemma 5.1]{Todorcevic2011}.
\begin{lemma}
    Suppose $\mathcal{G} = (X,E)$ is an open graph on a separable metric space $(X,d)$ which is not countably chromatic and that $n$ is a positive integer. Let $W \in \mathcal{H}^n(\mathcal{G})$ and let
$$
\partial W := \{\overline{x}\in W : (\forall \epsilon > 0)(\exists \overline{y}\in W )(\forall i < n)[(x_i,y_i) \in E \wedge d(x_i,y_i) < \epsilon]\}
$$
Then $W \setminus \partial W \notin \mathcal{H}^n(\mathcal{G})$.
\end{lemma}
\begin{definition}

Let $\theta$ be a regular cardinal such that $(X,d), \mathcal{G}, \mathcal{I}$ and $T^*$ belong to $H(\theta)$ and fix $<_w$ a well-order of $H(\theta)$.
\begin{enumerate}
    \item[(i)]  Let $X \in H(\theta)$, by $\mathcal{SK}(X)$ we denote the Skolem closure of $X$ (where the set of Skolem functions is defined using the  well-order $<_w$ of $H(\theta)$).
\item[(ii)]  If $M \prec H(\theta)$ is countable, by $M^+$ we denote $\mathcal{SK}(M \cup \{M\})$.
\end{enumerate}
\end{definition}
We are now ready to introduce our forcing notion. It is worth mentioning that while the idea of using successors of models in side conditions is a quite natural way to prove $(N,\P,T^*)$-genericity, it has been brought to our attention that this notion has already been considered by \cite{KuzeljavicTodorcevic2020} in a similar context.
\begin{definition}
Let $\P_{\mathcal{G}}$ be the collection of all pairs $p=(\mathcal{N}_p,f_p)$ satisfying the following conditions:
\begin{enumerate}[{\rm (i)}]
    \item $\mathcal{N}_p=\{N_0,\dots,N_{m}\}$ has the following properties:
    \begin{enumerate}[{\rm (a)}]
        \item For each $i\leq m$, $N_i \prec(H(\theta),\in,<_w)$ (countable)  containing $(X,d), \mathcal{G}, \mathcal{I}$ and $T^*$ as elements.

    \item $N_i\in N_i^+\in N_{i+1}$ for $i<m$.
    \end{enumerate}
\item $f_p:\mathcal{N}_p\to X$ such that
\begin{enumerate}[{\rm (a)}]
    \item For each $i\leq m,$ $f_p(N_i)\in N_{i+1}\setminus N_i^+$ (where $N_{m+1}=X)$.
    \item For each $i\leq m$, $f_p(N_i)\notin \bigcup (\mathcal{I}\cap N_i^+)$.
    \item For each $i<j\leq m$, $(f_p(N_i),f_p(N_j))\in E.$
\end{enumerate}
\end{enumerate}
Let $q\leq p$ if $f_p\subseteq f_q$.
\end{definition}
The next Lemma is straightforward and appears implicit in \cite{Todorcevic2011}. We just write it to have it for future reference.
\begin{lemma}
    Let $\{N_0,\dots,N_{m}\}$ be an increasing $\in$-chain of countable elementary submodels of $H(\theta)$ containing all relevant parameters and let $(x_0,\dots,x_{m})\in X^{m+1}$ be such that $x_i\in N_{i+1}\setminus N_i$ (where $N_{m+1}=X)$ and $x_i\notin \bigcup (N_i\cap \mathcal{I})$ for $i\leq m.$ If $W\subseteq X^{m+1}$ is such that $W\in N_0 $ and $(x_0,\dots,x_m)\in W$, then $W\in \mathcal{H}^{m+1}.$
\end{lemma}

\begin{theorem}
    The axiom  $\PFA(T^*)$ implies $\OGA$.
\end{theorem}

\begin{proof}
    We shall prove that $\P_\mathcal{G}$ is $T^*$-proper and adds an uncountable subset $Y$ of $X$ with $[Y]^2\subseteq E$.

    Let $\theta^*$ be a large enough regular cardinal such that $\P_\mathcal{G}, T^*, H(\theta), <_w, \mathcal{I},\mathcal{H}, (X,d)$  and $\mathcal{G} $ belong to $H(\theta^*)$. Let $\overline{M}\prec H(\theta)$ countable which contains all relevant parameters. Consider $p\in \overline{M}\cap \P_\mathcal{G}$. We will find $q\leq p$ which is $(M,\P_\mathcal{G},T^*)$-generic. Set  $M:=\overline{M}\cap H(\theta)$ and $\delta=M\cap\al$. Define $q=(\mathcal{N}_p\cup \{M\},f_p\cup\{(M,q(M))\})$ where  $q(M)$ is any element of $X\setminus \bigcup (\mathcal{I}\cap M^+)$ and such that $(q(M),p(N))\in E$ for any $N\in\mathcal{N}_p$. Let us see that such a $q(M)$ exists. Let $\overline{N}$ be the maximum element of $\mathcal{N}_p$. Fix a countable basis $\mathcal{B}$ of $(X,d)$ in $M$ and  pick a basic open set $U\in\mathcal{B}$ such that $p(\overline{N})\in U$ and $(p(N),x)\in E$ for all $N\in\mathcal{N}_p\cap \overline{N}$ and $x\in U$.
Let $Z$ be the set of all elements of $x\in U$ such that for any $y\in U$  $(x,y)\notin E$. Notice that $Z\in \overline{N}\cap \mathcal{I}$. It follows, from clause (ii)(b) of the definition of the forcing, that $p(\overline{N})\notin Z$. Pick $y\in U$ such that $(y,p(\overline{N}))\in E$. Since $E$ is open there is a $V\in \mathcal{B}$ such that $y\in V\subseteq U$ and $V\times \{x\}\subseteq E.$ Let $q(M)$ be any element in $V\setminus (\bigcup (M^+\cap \mathcal{I})\cup M^+)$ (this is possible since $\mathcal{B}\subseteq \mathcal{H})$.

We claim that $q$ is a $(M,\P_\mathcal{G},T^*)$-generic. Fix a dense open set $D\in M$, $a\in T^*_\delta $ and $\dot A\in M$ a $\P_\mathcal{G}$-name for a subset of $T^*$ and a relation $R\in\{<_{T^*},<_{T^*}^-\}$. Let $r\leq q$. By extending $r$ further, if necessary, we may assume that $r \in D$, $r$ decides whether or not $a \in\dot A$. If $r$ forces that $a \notin\dot A$, then replace $\dot A$ in what follows by the canonical $\P_\mathcal{G}$-name for $T^*$. Thus, we may assume without loss of generality that $r$ forces that $a \in \dot A$.

We need to find $s\in D\cap M$ and $bRa$ such that $s$ is compatible with $r$ and $s\Vdash b\in\dot A$.  Let $r_M:=(\mathcal{N}_r\cap M, f_r\restriction\mathcal{N}_r\cap M).$ It follows that $r_M\in \P_\mathcal{G}\cap M$ and $r$ is an end-extension of $r_M$. Let $m=|\mathcal{N}_r\setminus M|$ and $M=N_0,\dots,N_{m-1}$ denote the increasing enumeration of  $\mathcal{N}_r\setminus M$. Let $W$ be the set of all $m$-tuples $(x_0,\dots,x_{m-1}) $ such that there is an end-extension $t$ of $r_M$ in $D$ such that
\begin{enumerate}[{\rm(a)}]
    \item
    $|\mathcal{N}_t\setminus\mathcal{N}_{r_M}|=m$
    \item if $K_0,\dots,K_{m-1}$ is the increasing enumeration of $\mathcal{N}_t\setminus\mathcal{N}_{r_M}$ then $t(K_i)=x_i$ for $i<m$.
    \item $t\Vdash a\in \dot A$.
\end{enumerate}
Using that $E$ is open, fix basic open sets $U_i (i<m)$ in $\mathcal{B}$ such that $r(N_i)\in U_i$ and $U_i\times U_j\subseteq E$ for $i\ne j \in m.$
Here is the key point of the definition of the forcing, since $a\in N_0^+$, then $W\in N_0^+$ and $(r(N_0),\dots,r_{m-1}(N_{m-1}))\in W$. Thus it follows, from Lemma 4.7 and Lemma 4.4, that $\partial W\in \mathcal{H}^m $ and $(r(N_0),\dots,r_{m-1}(N_{m-1}))\in \partial W.$ Pick $\varepsilon>0$ such that $B_{\varepsilon}(r(N_i))\subseteq U_i$ for all $i<m$. From the definition of $\partial W$ we infer that there is a tuple $(x_0,\dots,x_{m-1})\in W$ such that $(r(N_i),x_i)\in E$ and $d(p(N_i),x_i)<\varepsilon$ for all $i<m$. Using, once more,  that $E$ is open find basic open sets $V_i (i<m)$ such that $x_i\in V_i\subseteq U_i$ and $V_i\times \{r(N_i)\}\subseteq E$ for $i<m$. Let $t$ be  a witness that  $(x_0,\dots,x_{m-1})\in W$ (we do not claim that $t$ is compatible with $r$).

Let $B$ denote the set of nodes $c\in T^*$ such that there is an $s\in D$ such that $s\Vdash c\in \dot A$
and $(s(K_0),\dots,s(K_{m-1}))\in V_1\times\dots\times V_{m-1}$. Notice that $t$  witness that $a\in B$. Since $a$ is $(M,T^*)$-generic, there exists $bRa$ such that $b\in B$.  By elementarity, we can find such a $b$ in $M$. Let $s$ in $M$ be a witness that $b\in B$. We claim that $s$ and $r$ are compatible. Let $t:=s\cup r$. We will show that $t\in \P_\mathcal{G}$. Then clearly $t\leq r,s$. It is easy to check that $t$ satisfies all properties for being in $\P_\mathcal{G}$ except perhaps clause (ii)(c). To see this, it is sufficient to show that for any $i,j<m,$ $(r(N_i),s(K_j))\in E.$ We proceed by cases. On one hand, if $i\ne j$, then $(r(N_i),s(K_j))\in U_i\times U_j\subseteq E$. On the other hand, if $i=j$, then $(r(N_i),s(K_j))\subseteq \{r(N_i)\}\times V_i\subseteq E$. This completes the proof that $q$ is $(M,\P_\mathcal{G},T^*)$-generic.

Finally, observe that for each $\alpha<\al,$ the set $D_\alpha:=\{p: \alpha\in \bigcup \dom(p)\} $ is dense-open. Applying $\PFA(T^*)$, to $\P_\mathcal{G}$ and to the above sequence of dense open sets, we get a filter $G$ such that $Y:=\ran(\bigcup G)$ is an uncountable subset of $X$ which generates a complete subgraph of $\mathcal{G}.$ This concludes the proof of the Theorem.
\end{proof}

\section{Club isomorphisms}\label{sec:club-iso}

A further question posed in \cite{Krueger2020}, which is
attributed to Moore, is whether $\PFA(T^{*})$ implies that any two
normal almost Suslin trees are club-isomorphic. This is a natural
question, since a positive answer  would imply that the choice of $T^{*}$ is somewhat
canonical. We answer this negatively by showing that it is consistent with
$\PFA(T^{*})$ that not all normal almost Suslin trees are club-isomorphic.
As a counterpart to this ``negative'' result, we show that
$\PFA(T^{*})$ does imply that any two normal special Aronszajn trees are
club-isomorphic.

Let us begin with the ``positive'' result. Recall that an Aronszajn tree is called \emph{normal} if it has a
unique minimal element, it is Hausdorff, and every node has
$\aleph_{0}$ extensions in any level above it. Two $\omega_{1}$-trees
$T$ and $S$ are called \emph{club-isomorphic} if for some club $C \seb
\omega_{1}$, $T \res C \cong S \res C$. Abraham and Shelah
\cite{ShelahAbraham} showed that for any two normal Aronszajn trees
there is a proper forcing that introduces a club isomorphism between
them, therefore $\PFA$ implies that all normal Aronszajn trees are
club-isomorphic. Abraham-Shelah's forcing
is defined as follows:
Fix $T$ and $S$ normal Aronszajn trees, and let $\P(T,S)$  be the
forcing notion consisting of all the pairs $p = (A_{p},f_p)$ where $A_{p} \in
[\omega_{1}]^{<\omega}$, and $f_{p}$ is a finite function such that
$\dom(f_{p})$ is a downwards closed subtree of $T \res A_{p}$
in which every maximal branch has the same length,
$\ran(f_{p})$ is a downwards closed subtree of $S \res A_{p}$,
$f_{p}$ is increasing, and for every $x \in T$, $\hht_{T}(x) = \hht_{S}(f(x))$.
Let $(A_{q},f_{q}) \le (A_{p},f_{p})$ if $A_{q} \supseteq A_{p}$ and
$f_{q} \supseteq f_{p}$. The following
proposition can be found in \cite{Krueger2025}, and it is also implicit
in \cite{ShelahAbraham}.

\begin{proposition}\label{lem:compat-cond}
  Let $p,q \in \P(T,S)$ and $\alpha < \beta < \omega_{1}$ be such that:
  \begin{itemize}\itemsep0em
    \item $\alpha \in A_{p}$ and $\beta \in A_{q}$.
    \item $A_{p} \seb \beta$ and $A_{p} \cap \alpha = A_{q} \cap \beta$.
    \item $p \res \alpha = q \res \beta$.
    \item Every two elements $\dom(f_{q}) \cap T_{\beta}$
    split before $\alpha$ and every two elements of
    $\ran(f_{q}) \cap S_{\beta}$ split before $\alpha$.
    \item Every element of $\dom(f_{p}) \cap T_{\alpha}$ is incomparable
    with every element of $\dom(f_{q}) \cap T_{\beta}$, every element
    of $\ran(f_{p}) \cap S_{\alpha}$ is incomparable with
    every element of $\ran(f_{q}) \cap S_{\beta}$.
  \end{itemize}
  Then $p$ and $q$ are compatible.
\end{proposition}

\begin{lemma}
  If $T$ and $S$ are special Aronszajn trees,
  then $\P(T,S)$ is $T^{*}$-proper.
\end{lemma}
\begin{proof}
  Let $\P := \P(T,S)$.
  Let $g : T \cup S \to \omega$ witness that both $T$ and $S$
  are special.
  Let $\theta$ be a large enough regular cardinal and let $N \prec H(\theta)$ be
  countable such that $T,S,\P,g \in N$. Fix $p \in \P \cap N$
  and define $\delta := N \cap \omega_{1}$. Let
  $q \le p$ be any extension such that $A_{q} := A_{p} \cup \{\delta\}$. We claim that $q$ is $(N,\P,T^*)$-generic. To see this, fix a dense set $D\in M$, $a\in T^*_\delta $, $\dot A\in M$ a $\P$-name for a subset of $T^*$, and a relation $R\in\{<_{T^*},<_{T^*}^-\}$. Let $r\leq q$. By extending $r$ further, if necessary, we may assume that $r \in D$, $r$ decides whether or not $a \in\dot A$. If $r$ forces that $a \notin\dot A$, then replace $\dot A$ in what follows by the canonical $\P$-name for $T^*$. Thus, we may assume without loss of generality that $r$ forces that $a \in \dot A$.

We need to find $s\in D\cap M$ and $b \mathrel{R} a$ such that $s$ is compatible with $r$ and $s\Vdash b\in\dot A$.

  Let $n = |\dom(f_{r}) \cap T_{\delta}|$,
  and $(\hat{t}_{0},\dots,\hat{t}_{n-1})$ be some enumeration
  of $\dom(f_{r}) \cap T_{\delta}$. For each
  $i < n$, let
  $p_{i} := g(\hat{t}_{i})$ and $q_{i} := g(f_{r}(\hat{t}_{i}))$.
  Let $r_{0} := r \res \delta$,
  and let $\alpha$ be such that
  $\max(A_{r_{0}}) < \alpha < \delta$, and such that
  for all $i \neq j < n$, both $\hat{t}_{i}$ with $\hat{t}_{j}$
  and $\hat{s}_{i}$ with $\hat{s}_{j}$ split before $\alpha$.
  For every $i < n$, let $\ck{t}_{i} \in T_{\alpha}$ and
  $\ck{s}_{i} \in S_{\alpha}$ be such that $\ck{t}_{i} <_{T} \hat{t}_{i}$ and
  $\ck{s}_{i} <_{S} f_{r}(\hat{t}_{i})$.
  Now consider the set $C$ of all tuples
  $(z,s,\beta,t_{0},\dots,t_{n-1}) \in T^{*} \times \P \times
  \omega_{1} \times T^{n}$ such that:
  \begin{itemize}\itemsep0em
    \item $s \in D$ and $s \forces z \in \dt{A}$.
    \item $A_{s} \setm \alpha \neq \es$ and $\min(A_{s} \setm \alpha) =
    \beta$.
    \item $|\dom(f_{s}) \cap T_{\beta}| = n$.
    \item $\dom(f_{s}) \cap T_{\beta} = \{t_{i} : i < n\}$.
    \item For all $i \neq j < n$, $\ck{t}_{i}$ and $t_{j}$
    are $T$-incomparable,
    and $\ck{s}_{i}$ and $f_{s}(t_{j})$ are $S$-incomparable.
    \item For all $i < n$, $g(t_{i}) = p_{i}$ and
    $g(f_{s}(t_{i})) = q_{i}$.
  \end{itemize}
  Note that $(x,r,\delta,\hat{t}_{0},\dots,
  \hat{t}_{n-1}) \in C$. Thus
  $B:= \{z : (\exists s,\beta,\vec{t})((z,s,\beta,\vec{t})
  \in C)\}$
  is nonempty, it is in $N$, and $x \in B$.
  Since $x$ is
  $(N,T^{*})$-generic, there exists $y \in B$ such that
  $y \mathrel{R} x$, so $(y,s,\beta,t_{0},\dots,t_{n-1}) \in C$
  for some $s,\beta,t_{0},\dots,t_{n-1} \in N$.
  Now note that for every $i,j < n$,
  $t_{i}$ is $T$-incomparable with $\hat{t}_{j}$ and
  $f_{s}(t_{i})$ is $S$-incomparable with $f_{r}({\hat{t}}_{j})$;
  if $i \neq j$ this is immediate from the definition of $C$,
  and if $i = j$, then it follows from $\beta < \delta$,
  $g(t_{i}) = g(\hat{t}_{i})$ and $g(f_{s}(t_{i})) = g(f_{r}(\hat{t_{i}}))$,
  again by the definition of $C$.
  Thus Lemma~\ref{lem:compat-cond} implies that
  $s$ is compatible with $r$. By definition of $C$,
  $s \forces y \in \dt{A}$ and $s \in D$. This finishes the proof.
\end{proof}

\begin{corollary}\label{cor:special-iso}
  $\PFA(T^{*})$ implies that
  all special normal Aronszajn trees are club-isomorphic.
\end{corollary}

\subsection{Not club-isomorphic almost Suslin trees}
We dedicate this section to show that, for some carefully chosen $T^*$, there is a model  of $\PFA(T^{*})$ in which
not all normal almost Suslin trees are club-isomorphic.  Let us fix
some notation first. If $T$ and $S$ are trees, $S + T$ denotes the
tree given by putting $S$ and $T$ side by side, and $S \otimes T$ the
set $\{(s,t) \in S \times T : \hht_{S}(s) = \hht_{T}(t)\}$ under the
product order, which is again a tree, and it is Aronszajn whenever
both $T$ and $S$ are $\omega_{1}$-trees and at least one of them is
Aronszajn.

Our strategy to obtain the consistency of
$\PFA(T^{*}) \; +$
``not all normal almost Suslin trees are club-isomorphic''
relies on the following lemma.

\begin{lemma}\label{lem:con-reduction}
  Assume there is a supercompact cardinal.
  Assume there are two normal almost Suslin trees, $T$ and $S$, such that
  $T \otimes T$ is again almost Suslin but $S \otimes S$ is not.
  Let $T^*:=(T\otimes T)+S$. Then there is a model of
  $\PFA(T^{*})$ in which $S$ and $T$ are not club isomorphic.
\end{lemma}
\begin{proof}
  Suppose that we are in a model in which there are $T$ and $S$
  as described. Let $T^{*} := (T \otimes T) + S$, clearly $T^{*}$ is also
  almost Suslin. Now construct a model of $\PFA(T^{*})$ by iterating
  $T^{*}$-proper forcings (using a Laver function as a bookkeping device)
  as in \cite{Krueger2020}. Since $T^{*}$ remains
  almost Suslin, so do $T$, $S$ and $T \otimes T$, and since stationary
  sets are preserved, $S \otimes S$ is still not almost Suslin. Thus $S$
  and $T$ cannot be club-isomorphic.
\end{proof}

\begin{rmk}\label{rmk:sum-is-enough}
  The tree we will use for $T$ will be a coherent tree. By
  Corollary~\ref{cor:coherent-asuslin} and the fact that coherence is
  preserved by generic extensions preserving $\omega_{1}$, we see that for such
  a $T$ it is enough to consider $T^{*} := T + S$, rather then $T^{*} :=
  (T \otimes T) + S$, in the previous proof.
\end{rmk}

We now show how to get a model in which there are $T$ and $S$
as in Lemma~\ref{lem:con-reduction}.

As mentioned, $T$ will be
a \emph{coherent tree}: an uncountable $T \seb
\omega^{<\omega_{1}}$ that is closed under taking initial segments and
such that for all $t,s \in T$, if $\dom(t) \le \dom(s)$ then $\{\xi <
\dom(t) : t(\xi) \neq s(\xi)\}$ is finite.

\begin{lemma}\label{lem:antichains}
  Let $T$ be a coherent Aronszajn tree,
  $\Gamma \seb \omega_{1}$ a stationary set, and
  $\{(x_{\alpha},y_{\alpha}) :\alpha \in \Gamma\}$ be such that
  $\dom(x_{\alpha}) = \dom(y_{\alpha}) = \alpha$
  for all $\alpha \in \Gamma$. Then there is stationary $\Gamma' \seb \Gamma$
  such that for all $\alpha < \beta$ in $\Gamma'$,
  $x_{\alpha} \tleq x_{\beta}$ iff $y_{\alpha} \tleq y_{\beta}$.
\end{lemma}
\begin{proof}
  We may assume that $\Gamma$ consists of limit ordinals.
  By coherence, for each $\alpha \in \Gamma$ there is
  $\gamma(\alpha) < \alpha$ such that
  $x_{\alpha} \res [\gamma(\alpha),\alpha) =
  y_{\alpha} \res [\gamma(\alpha),\alpha)$.
  Using Fodor's lemma, there
  is a stationary set $\Gamma' \seb \Gamma$ and
  $\gamma < \min(\Gamma')$ such that $\gamma(\alpha) = \gamma$ for
  all $\alpha \in \Gamma'$. Since $T_{\gamma}$ is countable,
  there are $x,y \in T_{\gamma}$ such that
  $\Gamma'' :=\{\alpha \in \Gamma' : x \tleq x_{\alpha}, y \tleq y_{\alpha}\}$
  is stationary. Then $x^{\frown}t \mapsto y^{\frown}t$ is a $\tleq$-isomorphism
  from $\{x_{\alpha} : \alpha \in \Gamma''\}$ onto
  $\{y_{\alpha} : \alpha \in \Gamma''\}$
\end{proof}

Note that the previous lemma
easily generalize to  finite tuples,
and thus the following Corollary also holds for any finite product $T
\otimes \cdots \otimes T$.

\begin{corollary}\label{cor:coherent-asuslin}
  If $T$ is a coherent Aronszaj tree, then
  $T$ is almost Suslin iff $T\otimes T$ is almost Suslin.
\end{corollary}

\begin{definition}
  For a tree $S$, its \emph{derived trees of dimension $n$}
  are the trees of the form $S(x_{1}) \otimes
  \dots S(x_{n})$ where $(x_{1},\dots,x_{n})$ are distinct elements of
  the same height, and $S(x_{i}) := \{s \in S : x_{i} \le_{S} s\}$.
\end{definition}

It is a standard fact that $\diamondsuit$ implies the existence
of a coherent Suslin tree. See for example \cite{Todorcevic1984}.
By Corollary~\ref{cor:coherent-asuslin},
to obtain our desired model it is enough find a model of
$\diamondsuit$ plus the existence of a Suslin tree
$S$ such that all its derived trees of dimension $2$ are special;
this implies that $S \otimes S$ is not almost Suslin.
One way to obtain such a model is to start with a model
of $\diamondsuit$ and use Krueger's forcing
\cite[Theorem 4.4]{Krueger2020suslin} to introduce such a tree $S$.
This works because Krueger's forcing has size $\omega_{1}$ (in a model
of $\CH$), and thus it preserves $\diamondsuit$. Other option,
which seems more natural to us, is to show that $\diamondsuit$
already implies the existence of a tree like $S$. We take
this approach.

Jensen and Johnsbr{\r a}ten \cite{Jen} used $\diamondsuit$ to
construct a Suslin tree $S$ with the property that $\{(s,t) \in S
\otimes S : s \neq t\}$ is special, as mentioned in \cite[9.15
(vi)]{Todorcevic1984}.  In particular, it follows that every derived
tree of dimension at least $2$ is special.  However, their original
proof is quite involved, and does not explicitly stablish this
fact. For clarity and completeness, we present a variant of their
construction. Our approach yields a proof of a stronger statement,
confirming a conjecture described as ``quite plausible'' by
Scharfenberger-Fabian in \cite[Sec3.6]{Scharfenberger2010}.

\begin{theorem}\label{thm:diamond-S}
  Assume $\diamondsuit$. For every $1 \le d < \omega$,
  there is a normal Suslin tree
  $S$ such that:
  \begin{itemize}
    \item All its derived trees of dimension at most $d$ are
    Suslin.
    \item All its derived trees of dimension greater than $d$
    are special.
  \end{itemize}

\end{theorem}
\begin{proof}
  Fix $d \ge 1$. First observe  that it is enough to make sure that all derived
  trees of dimension exactly $d$ are Suslin, and all derived trees
  of dimension exactly $d+1$ are special.

  Let us fix some notation. For $1 \le n < \omega$, let
  $S^{(n)}$ denote the $\bar{x} \in \bigotimes_{i < n}S$ that
  are injective, i.e., without repetitions. Similarly
  define $S^{(n)}_{\alpha}$. We also understand that
  if $\bar{x} \in S^{(n)}$, then $\bar{x} = (x_{0},\dots,x_{n-1})$.
  We order $S^{(n)}$ with the partial order induced
  from $\bigotimes_{i<_{n}}S$, which we will denote simply
  by $<_{S}$.

  Fix $\tup{A_{\delta} : \delta < \omega_{1}}$ a $\diamondsuit$-sequence
  for subsets of $\omega_{1}^{d}$.
  We will construct $S$ by recursion on the levels, making
  sure that for each $\alpha < \omega_{1}$,
  $S\res \alpha$ is a countable ordinal as a set.
  Together with $S$ we will define
  $f : \bigotimes_{i < d+1}S \to \Q_{\ge 0}$ satisfying the following:
  \begin{enumerate}[(i)]
    \item $f(\bar{x}) = 0$ iff $\bar{x}$ is not injective, i.e.,
    $\bar{x} \notin S^{(d+1)}$.
    \item For all $\bar{x},\bar{y} \in S^{(d+1)}$,
    if $\bar{x} <_{S} \bar{y}$ then $f(\bar{x}) < f(\bar{y})$.
    \item Suppose $\alpha < \beta$, and
    $\bar{x} \in S_{\alpha}^{(n)}$ with $n \ge d$. Then for all
    $I \in [n]^{d}$ and $\sigma : I \to S_{\beta}$ such that
    for all $i \in I$, $x_{i} <_{S} \sigma(i)$, and for
    all $\varepsilon \in \Q^{+}$, there is
    $\bar{y} \in S_{\beta}^{(n)}$ such that $\bar{x} <_{S} \bar{y}$,
    $y_{i} = \sigma(i)$ for $i \in I$,
    and for all $i_{0},\dots,i_{d} < n$ without
    repetitions,
    $f(x_{i_{0}},\dots,x_{i_{d}}) < f(y_{i_{0}},\dots,y_{i_{d}})
    < f(x_{i_{0}},\dots,x_{i_{d}}) + \varepsilon$.
  \end{enumerate}
  Note that property (ii) will ensure that all derived trees of
  dimension $d+1$ are special.

  Let $S_{0} := \{0\}$ and given
  $S_{\alpha}$, construct $S_{\alpha+1}$ by adding $\aleph_{0}$
  immediate successors to each node in $S_{\alpha}$ so that
  (i)--(iii) are satisfied. Now
  suppose $S \res \delta$
  has been defined for some countable limit ordinal satisfying (i)--(iii),
  and all the normality requirements. We will construct
  a countable set $B$ of cofinals branches through $S \res \delta$
  satisfying the following:
  \begin{itemize}
    \item[(B1)] For every $x \in S \res \alpha$ there is $b \in B$ such
    that $x \in B$.
    \item[(B2)] If it happens that $A_{\delta}$ is a maximal
    antichain of $S(x_{0}) \otimes \cdots \otimes S(x_{d-1})$ for
    some $\bar{x} \in (S\res \delta)^{(d)}$, then for every
    $b_{0},\dots,b_{n-1} \in B$ such that for all $i < d$,
    $x_{i} \in b_{i}$, $\bigotimes_{i < d}b_{i} \cap A_{\delta} \neq \es$.
  \end{itemize}
  We will define $S_{\delta}:= \{\tilde{b} : b \in B\}$ where for every
  $b \in B$, $\tilde{b}$ is a node above each element of $b$.
  This will ensure that the resulting tree is normal,
  and has every derived tree of dimension $d$ Suslin. Thus,
  it is enough to see how to construct $B$ satisfying (B1) and (B2),
  and how to extend $f$ on $S_{\delta}$ satisfying (i)--(iii).

  Define $\P_{\delta}$ to be the poset given by the
  $(\bar{a},h)$ such that for some $n \in \omega$,
  \begin{itemize}
    \item $\bar{a} \in \bigotimes_{i < n}S$. We call $n$ its \emph{length}
    and denote it by $|\bar{a}|$.
    Observe that we do not require $\bar{a}$ to be injective.
    \item $h : n^{d} \to \Q^{+}$ is such that for all
    $(i_{0},\dots,i_{d}) \in \dom(h)$,
    $f(a_{i_{0}},\dots,a_{i_{d}}) < h(i_{0},\dots,i_{d})$.
  \end{itemize}
  The order is given by $(\bar{b},g) \le (\bar{a},h)$ if
  $|\bar{b}| \ge |\bar{a}|$, $\hht(\bar{b}) \ge \hht(\bar{a})$,
  for all $i < |a|$, $b_{i} \ge_{S} a_{i}$, and
  $g \supseteq h$.

  Let $N \prec H(\omega_{2})$ be countable and such that
  $\delta,S\res\delta,A_{\delta},\P_{\delta} \in N$. Let
  $G$ be a $\P_{\delta}$-generic filter for $N$. For each
  $i < \omega$ let $b_{i} := \{a_{i} : (\bar{a},h) \in G, |\bar{a}| > i\}$.

  We claim that for each $i < \omega$, $b_{i}$ is a cofinal branch
  through $S \res \delta$, and moreover, if
  $B := \{b_{i} : i< \omega\}$, then  we have that (B1) and (B2) hold.
  To see that each $b_{i}$ is a cofinal branch,
  note that for each $i < \omega$
  and $\alpha < \delta$, the set of $(\bar{a},h) \in \P_{\delta}$ such that
  $|\bar{a}| > i$ and $\hht(\bar{a}) \ge \alpha$, is dense and it
  is definable from parameters in $N$, thus it is in $N$, so
  $G$ meets it. Since the  verification that the  clauses  (B1) and (B2) hold, are similar.
  We will only present the proof that (B2) holds. Suppose that $A_{\delta}$ is a maximal
  antichain of $\bigotimes_{i < d}S(x_{i})$, where
  $\bar{x} \in (S\res \delta)^{(n)}$. For each $x_{i}$, let
  $b_{k(i)}$ be such that $x_{i} \in b_{k(i)}$. Observe that
  $k$ is injective since $\bar{x}$ is injective. Now consider the
  set of $(\bar{a},h) \in \P_{\delta}$ such that
  $|\bar{a}| > \sup_{i < d}k(i)$,
  and either $(a_{k(0)},\dots,a_{k(d-1)}) \in A_{\delta}$ or
  it does not extend $\bar{x}$. Then this set is dense by (iii) in
  $S \res \delta$,
  and it is in $N$, thus $G$ meets it. This easily implies what we want.

  It remains to define $f$ on $S_{\delta}$. Call
  $\tilde{b}_{i}$ the node in $S_{\delta}$ above $b_{i}$.
  For every injective $(i_{0},\dots,i_{d})$ we define
  $f(\tilde{b}_{i_{0}},\dots,\tilde{b}_{i_{d}}) := h(i_{0},\dots,i_{d})$
  where $(\bar{a},h) \in G$ is any such that $|\bar{a}| > \sup_{i \le d}i_{d}$.
  This is well defined because $G$ is a filter. Of course if $(i_{0},\dots,i_{d})$
  is not injective we let $f(\tilde{b}_{i_{0}},\dots,\tilde{b}_{i_{d}}) = 0$.
  It follows
  from the definition of $\P_{\delta}$ that this will satisfy
  (i) and (ii). We need to show (iii).

  Pick $\bar{x} \in (S\res \delta)^{n}$. We may assume $n \ge d+1$,
  since otherwise (iii) holds vacuously. Pick $I \in [n]^{d}$,
  $\sigma : I \to \omega$ and $\varepsilon \in \Q^{+}$.
  Consider $D \seb \P_{\delta}$ the set of $(\bar{a},h)$ satisfying that
  \begin{itemize}
    \item $\ran(\sigma) \seb |\bar{a}|$, $\hht(\bar{a}) > \hht(\bar{x})$ and
    $\bar{a}$ is injective.
    \item If it is the case that for all $i \in I$, $x_{i} <_{S} \sigma(i)$,
    then there is $\sigma' : n \to \omega$ such that:
    \begin{itemize}
      \item $\ran(\sigma') \seb |\bar{a}|$ and $\sigma' \res I = \sigma$.
      \item For all injective $(i_{0},\dots,i_{d}) \in n^{d+1}$,
  $h(a_{\sigma'(i_{0})},\dots,a_{\sigma'(i_{d})}) < f(x_{i_{0}},\dots,x_{i_{d}}) + \varepsilon$.
    \end{itemize}
  \end{itemize}
  An application if (iii) on $S \res \delta$ shows that $D$ is dense.
  This easily implies that (iii) holds on $S \res (\delta +1)$ so
  we are done.
\end{proof}

By Lemma~\ref{lem:con-reduction} and Remark~\ref{rmk:sum-is-enough}
we obtain the following.

\begin{corollary}\label{cor:con-non-iso}
  Assuming $\diamondsuit$ there is a Suslin tree $T^{*}$ and a model of $\PFA(T^{*})$ which contains two normal
  almost Suslin trees that are not club-isomorphic.
\end{corollary}

\section{Concluding remarks and questions}\label{sec:questions}

Can Corollary~\ref{cor:special-iso} be improved to ``$\PFA(T^{*})$ implies
that all normal Aronszajn trees without almost-Suslin
subtrees are club-isomorphic'' for some reasonable definition of
subtree? This is motivated by Krueger's result saying that under
$\PFA(S)$ all normal Aronszajn trees without Suslin subtrees are
club-isomorphic (see \cite{Krueger2025}). Note that a positive answer
would be somewhat optimal because of the following fact: if there is an
almost-Suslin tree, then there is a family of size $2^{\aleph_{1}}$ of
pairwise not club-isomorphic normal Aronszajn trees, each of which has
a stationary antichain but contains an almost Suslin subtree
(see \cite[5.15 (v)]{Todorcevic1984}).

In view of Corollary~\ref{cor:con-non-iso}, is it consistent to have an
an almost Suslin tree $T^{*}$ such that after forcing to
obtain $\PFA(T^{*})$ all almost Suslin trees are club-isomorphic?

\bibliographystyle{alpha}

\end{document}